\documentclass[11pt]{amsart}

\usepackage[T1]{fontenc}
\usepackage[utf8]{inputenc}
\usepackage{lmodern}
\usepackage{amsmath,amssymb,amsthm,mathtools}
\usepackage{geometry}
\usepackage{booktabs}
\usepackage{microtype}
\usepackage{hyperref}
\usepackage{xcolor}

\newtheorem{theorem}{Theorem}[section]
\newtheorem{proposition}[theorem]{Proposition}
\newtheorem{corollary}[theorem]{Corollary}
\newtheorem{lemma}[theorem]{Lemma}
\theoremstyle{definition}
\newtheorem{definition}[theorem]{Definition}
\theoremstyle{remark}
\newtheorem{remark}[theorem]{Remark}

\newcommand{\N}{\mathbb N}
\newcommand{\K}{\mathbb K}
\newcommand{\Lcal}{\mathcal L}
\newcommand{\as}{\mathrm{as}}
\newcommand{\mult}{\mathrm{mult}}

\newcommand{\doi}[1]{\href{https://doi.org/#1}{doi:#1}}
\newcommand{\norm}[1]{\left\lVert #1\right\rVert}

\title[Optimality in a Multilinear Extension of Kwapie\'n's Theorem]
{Optimality in a Multilinear Extension of Kwapie\'n's Theorem}
\author[L.~M.~Casta\~no-Mar\'in]{Luis Miguel Casta\~no-Mar\'in}
\address{Departamento de Matem\'aticas\\
Universidad Nacional de Colombia\\
Bogot\'a, Colombia}
\email{lucastanom@unal.edu.co}

\author[G.~Ribeiro]{Geivison Ribeiro}
\address{Department of Mathematics\\
Universidade Federal de Sergipe\\
Campus Professor Alberto Carvalho, Av. Vereador Ol\'impio Grande, s/n,
49506-036 Itabaiana, SE, Brazil}
\email{geivison@mat.ufs.br}

\author[D.~Serrano-Rodr\'iguez]{Diana Serrano-Rodr\'iguez}
\address{Departamento de Matem\'aticas\\
Universidad Nacional de Colombia\\
Bogot\'a, Colombia}
\email{diserranor@unal.edu.co}

\subjclass[2020]{46A35, 46A45, 46B28, 46G25}
\keywords{Kwapie\'n's theorem, absolutely summing multilinear operators,
multiple summing operators, Walsh matrices, quasi-Banach interpolation}

\begin{document}

\begin{abstract}
Bayart, Pellegrino, and Rueda obtained diagonal absolute summability
exponents for continuous \(m\)-linear operators from \((\ell_1)^m\)
into \(\ell_p\). Their optimality was known for \(2\le p\le\infty\),
whereas the case \(1\le p<2\) remained open. We settle the latter
case by means of finite Walsh test operators, and consequently obtain
the optimal exponent for every \(1\le p\le\infty\). The same argument,
combined with the corresponding interpolation and embedding results,
also yields optimal Kwapie\'n-type statements for several related
range spaces.
\end{abstract}

\maketitle

\section{Introduction}

Grothendieck's theorem asserts that every bounded linear operator from
\(\ell_1\) into \(\ell_2\) is absolutely \((1,1)\)-summing, and the
exponent \(1\) is optimal. Kwapie\'n \cite{Kwapien} extended this
result to operators from \(\ell_1\) into \(\ell_p\): every continuous
linear operator
\[
T:\ell_1\longrightarrow\ell_p
\]
is absolutely \((r,1)\)-summing if and only if
\begin{equation}\label{eq:Kwapien-linear}
\frac1r\le 1-\left|\frac1p-\frac12\right|.
\end{equation}
Equivalently, the optimal exponent is
\begin{equation}\label{eq:linear-r}
r=
\begin{cases}
\dfrac{2p}{3p-2}, & 1\le p\le2,\\[8pt]
\dfrac{2p}{p+2}, & 2\le p\le\infty.
\end{cases}
\end{equation}

There are two natural multilinear versions of this question. For an
\(m\)-linear operator, one may sum over all multi-indices, leading to
multiple summability, or only over the diagonal, leading to absolute
summability. In the multiple setting, the linear Kwapie\'n's exponent is
preserved on products of \(\ell_1\); this follows from the coincidence
theorem of Botelho and Pellegrino \cite[Corollary~4.3]{BP2009}. Thus,
for multiple summability, the exponent does not depend on the degree
\(m\).

The diagonal problem has a different behavior. Bayart, Pellegrino, and
Rueda \cite{BPR} proved that
\[
\Lcal({}^m\ell_1;\ell_p)
=
\Lcal_{\as(r_{m,p};1)}({}^m\ell_1;\ell_p),
\]
where
\begin{equation}\label{eq:BPR-exponents}
r_{m,p}=
\begin{cases}
\displaystyle\frac{2p}{mp+2p-2}, & 1\le p\le2,\\[8pt]
\displaystyle\frac{2p}{mp+2}, & 2\le p\le\infty.
\end{cases}
\end{equation}
For \(m=1\), this recovers \eqref{eq:linear-r}. For \(m\ge2\), the
exponents in \eqref{eq:BPR-exponents} may be smaller than one, so the
relevant summability ideals are quasi-Banach.

The optimality of \eqref{eq:BPR-exponents} for \(2\le p\le\infty\) was
proved in \cite{NSS}. The purpose of the present paper is to settle the
remaining range \(1\le p<2\). We do so by constructing finite
\(m\)-linear test operators from Walsh characters on
\(\mathbb F_2^d\). Finite-dimensional constructions based on structured sign patterns are
often decisive in the determination of sharp exponents in multilinear
inequalities. Bennett's bilinear inequality and constructive estimates
for its constants provide one instance of this principle
\cite{PellegrinoRaposoBennett}. Related anisotropic sharp-exponent
phenomena for Hardy--Littlewood type inequalities were studied in
\cite{NunezSerranoTeixeira}. In the present work, Walsh operators are
used instead to detect the optimal coincidence exponents for absolutely
summing multilinear operators.

The same construction can be combined with standard interpolation and
finite-dimensional embedding arguments. We use this observation to
obtain corresponding optimal results for several related range spaces
and, together with the block reduction from \cite{NSS}, for the
block-summing formulation.

The Walsh construction gives the lower estimate in the unresolved range $1\le p<2$; the interpolation arguments below provide the corresponding upper estimates and the stated extensions.

\section{Preliminaries}

Throughout this paper, $\K \in \{\mathbb R, \mathbb C\}$, and $E, E_1, \ldots, E_m, F, W$ denote Banach spaces over $\K$. The space of continuous $m$-linear mappings from $E_1 \times \cdots \times E_m$ into $F$, equipped with the canonical operator norm $\|T\| := \sup\{\|T(x_1,\ldots,x_m)\|: x_k \in B_{E_k}\}$, is denoted by $\Lcal(E_1,\ldots,E_m;F)$, abbreviated as $\Lcal({}^mE;F)$ when $E_1 = \cdots = E_m = E$.

For $1 \le s \le \infty$, the weak $s$-norm of a finite family $(x_j)_{j=1}^N \subset E$ is defined as
\[
\|(x_j)_{j=1}^N\|_{w,s} := \sup_{\varphi \in B_{E^*}} \left(\sum_{j=1}^N |\varphi(x_j)|^s\right)^{1/s},
\]
with the usual supremum over $j$ when $s=\infty$.

\begin{definition}
Let $0 < r < \infty$ and $1 \le s_1,\ldots,s_m \le \infty$. An operator $T \in \Lcal(E_1,\ldots,E_m;F)$ is called \emph{absolutely $(r;s_1,\ldots,s_m)$-summing} if there exists a constant $C \ge 0$ such that, for every $N \in \N$ and all finite families $(x_j^{(k)})_{j=1}^N \subset E_k$ ($k=1,\ldots,m$),
\begin{equation}\label{eq:def-abs-summing}
\left(\sum_{j=1}^N \|T(x_j^{(1)},\ldots,x_j^{(m)})\|_F^r\right)^{1/r} \le C \prod_{k=1}^m \|(x_j^{(k)})_{j=1}^N\|_{w,s_k}.
\end{equation}
The least constant $C$ in \eqref{eq:def-abs-summing} is denoted by $\pi_{\as(r;s_1,\ldots,s_m)}(T)$, and the class of all such operators is denoted by $\Lcal_{\as(r;s_1,\ldots,s_m)}(E_1,\ldots,E_m;F)$. When $s_1 = \cdots = s_m = s$, we abbreviate this to $\Lcal_{\as(r;s)}$.
\end{definition}

\begin{definition}
Let $0 < r < \infty$ and $1 \le s_1,\ldots,s_m \le \infty$. An operator $T \in \Lcal(E_1,\ldots,E_m;F)$ is called \emph{multiple $(r;s_1,\ldots,s_m)$-summing} if there exists $C \ge 0$ such that, for every $N \in \N$ and all finite families $(x_{j_k}^{(k)})_{j_k=1}^N \subset E_k$ ($k=1,\ldots,m$),
\begin{equation}\label{eq:def-mult-summing}
\left(\sum_{j_1,\ldots,j_m=1}^N \|T(x_{j_1}^{(1)},\ldots,x_{j_m}^{(m)})\|_F^r\right)^{1/r} \le C \prod_{k=1}^m \|(x_j^{(k)})_{j=1}^N\|_{w,s_k}.
\end{equation}
The class of such operators is denoted by $\Lcal_{\mult(r;s_1,\ldots,s_m)}(E_1,\ldots,E_m;F)$.
\end{definition}

We will rely on the following uniform boundedness principle for quasi-Banach operator ideals (see \cite{KaltonPeckRoberts} and \cite{BotelhoMichelsPellegrino}).

\begin{lemma}[Uniform coincidence principle]\label{lem:uniform-coincidence}
Let $0 < r < \infty$, $1 \le s \le \infty$, and let $E_1,\ldots,E_m,F$ be Banach spaces. If
\[
\Lcal(E_1,\ldots,E_m;F) = \Lcal_{\as(r;s)}(E_1,\ldots,E_m;F),
\]
then there exists a constant $C > 0$ such that $\pi_{\as(r;s)}(T) \le C \|T\|$ for every $T \in \Lcal(E_1,\ldots,E_m;F)$.
\end{lemma}

\begin{proof}
For $r \ge 1$, the ideal $\Lcal_{\as(r;s)}$, equipped with the norm $\pi_{\as(r;s)}$, is a Banach space. For $0 < r < 1$, it is an $r$-Banach space under the $r$-norm $\pi_{\as(r;s)}(\cdot)^r$; see \cite{BotelhoMichelsPellegrino}. Thus, in either case, it is a complete metric topological vector space ($F$-space). Under the coincidence hypothesis, the identity mapping from $\Lcal_{\as(r;s)}$ onto $\Lcal(E_1,\ldots,E_m;F)$ is a continuous linear bijection. The open mapping theorem for $F$-spaces \cite[Chapter~1]{KaltonPeckRoberts} ensures that its inverse is continuous, yielding the asserted uniform estimate.
\end{proof}

\begin{definition}\label{def:crude-finite-representability}
Let $E$ and $W$ be Banach spaces. We say that $E$ is crudely finitely representable in $W$, and
write
\[
E <_f W,
\]
if there exists a constant $D\ge1$ such that, for every
finite-dimensional subspace $E_0\subset E$, there are a subspace
$W_0\subset W$ and an isomorphism
\[
J:E_0\longrightarrow W_0
\]
satisfying
\[
\|J\|\le1
\qquad\text{and}\qquad
\|J^{-1}\|\le D.
\]
In particular, \(\ell_s<_f W\) means that the spaces
\(\ell_s^n\) embed into \(W\) with a distortion bounded independently
of \(n\).
\end{definition}

\section{The Finite Walsh Construction and Optimality in \texorpdfstring{$\ell_p$}{lp}}\label{sec:Walsh}

In this section, we present the finite harmonic analysis construction that resolves the optimality of $r_{m,p}$ for $1 \le p < 2$.

Fix $d \in \N$. Let $G_d := \mathbb F_2^d = (\mathbb Z / 2\mathbb Z)^d$, so that $M := |G_d| = 2^d$. Group addition in $G_d$ is denoted by $\oplus$. We fix an enumeration of $G_d$, which identifies the sequence space $\ell_q(G_d)$ isometrically with $\ell_q^M$ for all $1 \le q \le \infty$.

For each $\alpha \in G_d$, the associated \emph{Walsh character} is defined by
\[
h_\alpha(\gamma) := (-1)^{\alpha \cdot \gamma}, \qquad \gamma \in G_d,
\]
where $\alpha \cdot \gamma = \sum_{j=1}^d \alpha_j \gamma_j \pmod 2$. The $M \times M$ matrix $H_M := (h_\alpha(\gamma))_{\alpha,\gamma \in G_d}$ is the classical Walsh--Hadamard matrix. Orthogonality of group characters yields
\begin{equation}\label{eq:Hadamard}
H_M H_M^T = M I_M.
\end{equation}

For families of vectors $a^{(1)}, \ldots, a^{(m)} \in \ell_1(G_d)$, we define their $m$-fold group convolution $\Phi_M$ by
\[
\Phi_M(a^{(1)}, \ldots, a^{(m)})(\gamma) := \sum_{\gamma_1 \oplus \cdots \oplus \gamma_m = \gamma} a^{(1)}(\gamma_1) \cdots a^{(m)}(\gamma_m).
\]

\begin{lemma}[Walsh identities and weak norms]\label{lem:Walsh-identities}
Let $M = 2^d$. Then:
\begin{enumerate}
\item[\textup{(i)}] For all $a^{(1)}, \ldots, a^{(m)} \in \ell_1(G_d)$,
\[
\|\Phi_M(a^{(1)}, \ldots, a^{(m)})\|_1 \le \prod_{k=1}^m \|a^{(k)}\|_1.
\]
\item[\textup{(ii)}] For every $\alpha \in G_d$,
\[
\Phi_M(h_\alpha, \ldots, h_\alpha) = M^{m-1} h_\alpha, \qquad H_M h_\alpha = M e_\alpha.
\]
\item[\textup{(iii)}] The family $(h_\alpha)_{\alpha \in G_d}$ satisfies
\[
\|(h_\alpha)_{\alpha \in G_d}\|_{w,1} \le M^{3/2} \quad \text{in } \ell_1(G_d).
\]
\end{enumerate}
\end{lemma}

\begin{proof}
Item (i) follows directly from applying the triangle inequality to the definition of $\Phi_M$. For (ii), note that $h_\alpha(\gamma_1) \cdots h_\alpha(\gamma_m) = h_\alpha(\gamma_1 \oplus \cdots \oplus \gamma_m)$. For any fixed $\gamma \in G_d$, there are exactly $M^{m-1}$ tuples $(\gamma_1,\ldots,\gamma_m) \in G_d^m$ such that $\gamma_1 \oplus \cdots \oplus \gamma_m = \gamma$ (the first $m-1$ elements can be chosen arbitrarily, uniquely determining the last). This establishes $\Phi_M(h_\alpha,\ldots,h_\alpha) = M^{m-1}h_\alpha$. The identity $H_M h_\alpha = M e_\alpha$ is the matrix formulation of character orthogonality.

Finally, for (iii), let $a \in B_{\ell_\infty(G_d)}$. Applying the Cauchy--Schwarz inequality followed by \eqref{eq:Hadamard},
\[
\sum_{\alpha \in G_d} |\langle a, h_\alpha \rangle| \le M^{1/2} \left(\sum_{\alpha \in G_d} |\langle a, h_\alpha \rangle|^2\right)^{1/2} = M^{1/2} \|H_M a\|_2 = M \|a\|_2 \le M^{3/2} \|a\|_\infty \le M^{3/2}.
\]
Taking the supremum over all $a \in B_{\ell_\infty(G_d)}$ completes the proof.
\end{proof}

\begin{proposition}[Walsh test operators]\label{prop:Walsh-test}
Let $m \in \N$, $1 \le p \le \infty$, and $M = 2^d$. There exists an $m$-linear operator $A_{M,p} \in \Lcal({}^m\ell_1^M; \ell_p^M)$ with $\|A_{M,p}\| \le 1$ such that, for every $\alpha \in G_d$,
\[
\|A_{M,p}(h_\alpha, \ldots, h_\alpha)\|_p = \begin{cases}
M^{m-1+1/p}, & 1 \le p \le 2,\\[4pt]
M^{m-1/p}, & 2 \le p \le \infty.
\end{cases}
\]
\end{proposition}

\begin{proof}
If $1 \le p \le 2$, we define $A_{M,p} := \Phi_M$, viewing the output in $\ell_p(G_d)$. Since $\|z\|_p \le \|z\|_1$, Lemma~\ref{lem:Walsh-identities}(i) yields $\|A_{M,p}\| \le 1$. By Lemma~\ref{lem:Walsh-identities}(ii),
\[
\|A_{M,p}(h_\alpha, \ldots, h_\alpha)\|_p = M^{m-1} \|h_\alpha\|_p = M^{m-1} M^{1/p} = M^{m-1+1/p}.
\]

If $2 \le p \le \infty$, define $W_{M,p} := M^{-1/p} H_M: \ell_1(G_d) \to \ell_p(G_d)$ (with $M^{-1/\infty} := 1$), and set $A_{M,p} := W_{M,p} \circ \Phi_M$. Since $\|H_M z\|_\infty \le \|z\|_1$ and $\|w\|_p \le M^{1/p} \|w\|_\infty$, we have $\|W_{M,p}\| \le 1$, and hence $\|A_{M,p}\| \le 1$. Finally,
\[
A_{M,p}(h_\alpha, \ldots, h_\alpha) = M^{-1/p} H_M(M^{m-1} h_\alpha) = M^{m-1/p} e_\alpha,
\]
whose $\ell_p$-norm is exactly $M^{m-1/p}$.
\end{proof}

\begin{theorem}[Sharpness of the BPR exponents on $\ell_p$]\label{thm:uniform-optimality}
Let $m \in \N$ and $1 \le p \le \infty$. If $\Lcal({}^m\ell_1;\ell_p) = \Lcal_{\as(r;1)}({}^m\ell_1;\ell_p)$ for some $r > 0$, then necessarily $r \ge r_{m,p}$. Consequently, the exponent $r_{m,p}$ in \eqref{eq:BPR-exponents} is optimal for all $1 \le p \le \infty$.
\end{theorem}

\begin{proof}
Assume the coincidence holds. By Lemma~\ref{lem:uniform-coincidence}, there is a constant $C > 0$ such that $\pi_{\as(r;1)}(T) \le C \|T\|$ for all $T \in \Lcal({}^m\ell_1;\ell_p)$. Let $P_M: \ell_1 \to \ell_1^M$ be the coordinate projection and $J_{M,p}: \ell_p^M \to \ell_p$ be canonical inclusion. We identify $A_{M,p}$ with $J_{M,p} \circ A_{M,p} \circ (P_M,\ldots,P_M) \in \Lcal({}^m\ell_1;\ell_p)$. Applying the summing inequality to $(h_\alpha)_{\alpha \in G_d}$ in $\ell_1$:

For $1 \le p \le 2$:
\[
M^{1/r + m - 1 + 1/p} = \left(\sum_{\alpha \in G_d} \|A_{M,p}(h_\alpha,\ldots,h_\alpha)\|_p^r\right)^{1/r} \le C \|A_{M,p}\| \|(h_\alpha)\|_{w,1}^m \le C M^{3m/2}.
\]
Since this holds for every $M = 2^d$, letting $M \to \infty$ forces
\[
\frac1r + m - 1 + \frac1p \le \frac{3m}{2} \iff \frac1r \le \frac m2 + 1 - \frac1p \iff r \ge \frac{2p}{mp+2p-2}.
\]

For $2 \le p \le \infty$:
\[
M^{1/r + m - 1/p} = \left(\sum_{\alpha \in G_d} \|A_{M,p}(h_\alpha,\ldots,h_\alpha)\|_p^r\right)^{1/r} \le C M^{3m/2},
\]
which yields $\frac1r \le \frac m2 + \frac1p \iff r \ge \frac{2p}{mp+2}$.
\end{proof}

\section{A scalar interpolation input}\label{sec:scalar}

We use the inner complex interpolation method for quasi-Banach lattices. The following finite-dimensional identity is standard; it is the only scalar quasi-Banach interpolation result used in the paper.

\begin{theorem}[Calder\'on--\v{S}estakov formula]\label{thm:scalar}
Let $N\in\N$, $0<r_0,r_1<\infty$, and $0<\theta<1$. Define $r$ by
\[
 \frac1r=\frac{1-\theta}{r_0}+\frac\theta{r_1}.
\]
Then
\[
 [\ell_{r_0}^N,\ell_{r_1}^N]_\theta=\ell_r^N
\]
isometrically.
\end{theorem}

\begin{proof}
This is the finite-dimensional case of the identification of inner complex interpolation with the Calder\'on product for quasi-Banach lattices. See \cite{Shestakov,Yuan}. Since the spaces involved are finite-dimensional, the closure of the intersection occurring in the general statement is the whole Calder\'on product.
\end{proof}

\section{Finite vector-valued interpolation}

\begin{lemma}\label{lem:finite-vector}
Let $(X_0,X_1)$ be a compatible couple of Banach spaces and put
$X_\theta=[X_0,X_1]_\theta$. If $0<r_0,r_1<\infty$ and
$r^{-1}=(1-\theta)r_0^{-1}+\theta r_1^{-1}$, then, for every $N\in\N$,
\[
 [\ell_{r_0}^N(X_0),\ell_{r_1}^N(X_1)]_\theta
 =\ell_r^N(X_\theta)
\]
isometrically.
\end{lemma}

\begin{proof}
The direct inequality $\norm{x}_{[\ell_{r_0}^N(X_0),\ell_{r_1}^N(X_1)]_\theta} \le \norm{x}_{\ell_r^N(X_\theta)}$ is constructive. Let $x=(x_1,\ldots,x_N)\in\ell_r^N(X_\theta)$, set $\rho_j=\norm{x_j}_{X_\theta}$, and put $\lambda=(\sum_{j=1}^N \rho_j^r)^{1/r}$ (omitting zero coordinates). For every $\varepsilon>0$, choose $g_j\in\mathcal F(X_0,X_1)$ with $g_j(\theta)=x_j/\rho_j$ and $\norm{g_j}_{\mathcal F(X_0,X_1)}\le 1+\varepsilon$. With $\alpha(z)=(1-z)/r_0 + z/r_1$, define
\[
 \varphi_j(z)=\lambda\left(\frac{\rho_j}{\lambda}\right)^{r\alpha(z)},
 \qquad F(z)=(\varphi_j(z)g_j(z))_{j=1}^N.
\]
Then $F(\theta)=x$, and on either boundary line $\Re z=k$, $k\in\{0,1\}$, one has $|\varphi_j(i+it)|^{r_i}=\lambda^{r_i}(\rho_j/\lambda)^r$. Consequently, $\sup_t \norm{F(it)}_{\ell_{r_0}^N(X_0)} \le (1+\varepsilon)\lambda$ and $\sup_t \norm{F(1+it)}_{\ell_{r_1}^N(X_1)} \le (1+\varepsilon)\lambda$. Letting $\varepsilon\downarrow0$ yields
\[
 \norm{x}_{[\ell_{r_0}^N(X_0),\ell_{r_1}^N(X_1)]_\theta} \le \norm{x}_{\ell_r^N(X_\theta)}. \tag{4.1}
\]
The reverse inequality follows from the general Calder\'on product identification for quasi-Banach lattices of vector-valued sequences. Indeed, since $N$ is finite, $\ell_r^N(X)$ coincides isometrically with the Calder\'on space $(\ell_{r_0}^N)^{1-\theta}(\ell_{r_1}^N)^\theta(X_\theta)$ (see Calder\'on \cite[Theorem~10.1]{Calderon1964} for Banach spaces, and \v{S}estakov \cite{Shestakov} and Yuan \cite[Theorem~3.4]{Yuan} in the quasi-Banach framework). Taking the infimum over admissible analytic functions in $\mathcal F(\ell_{r_0}^N(X_0),\ell_{r_1}^N(X_1))$ establishes the reverse inequality to (4.1).
\end{proof}

\section{The diagonal interpolation principle}\label{sec:diagonal}

The operator-space interpolation step is based on the canonical isometric 
identification
\[
\Lcal({}^m\ell_1; X) \cong \Lcal\big(\widehat\otimes_{\pi}^m\ell_1; X\big) 
\cong \Lcal(\ell_1(\N^m); X) \cong \ell_\infty(\N^m; X),
\]
given by $T \mapsto \big(T(e_{j_1},\ldots,e_{j_m})\big)_{(j_1,\ldots,j_m)\in\N^m}$. 
By classical vector-valued sequence space interpolation (see 
Bergh and L\"ofstr\"om \cite[Theorem~5.6.3]{BerghLofstrom}; see also 
Bayart, Pellegrino, and Rueda \cite[p.~6]{BPR}), one has
\[
[\ell_\infty(\N^m; X_0),\, \ell_\infty(\N^m; X_1)]_\theta 
= \ell_\infty(\N^m; [X_0,X_1]_\theta)
\]
isometrically. Consequently,
\begin{equation}\label{eq:operator-interpolation}
[\Lcal({}^m\ell_1; X_0),\, \Lcal({}^m\ell_1; X_1)]_\theta 
= \Lcal({}^m\ell_1; [X_0,X_1]_\theta)
\end{equation}
isometrically.

\begin{remark}[Real versus complex scalars]\label{rem:real-complex}
Proposition~\ref{prop:diagonal} is stated for complex Banach spaces, since its proof uses the complex interpolation method. For real spaces one may pass to compatible complexifications and then restrict the resulting estimate to the real subspaces. The norm of a multilinear complexification need not equal the norm of the original real operator; however, it is bounded by a constant depending only on the degree and on the chosen complexification. Hence this passage preserves the summability exponent, which is the only point used below.
\end{remark}
\begin{proposition}\label{prop:diagonal}
Let $(X_0,X_1)$ be a compatible couple of Banach spaces, $0<\theta<1$, and
\[
 \frac1{r_\theta}=\frac{1-\theta}{r_0}+\frac\theta{r_1},
 \qquad 0<r_0,r_1<\infty.
\]
Assume that, for $i=0,1$,
\[
 \Lcal({}^m\ell_1;X_i)=\Lcal_{\as(r_i;1)}({}^m\ell_1;X_i).
\]
Then
\[
 \Lcal({}^m\ell_1;[X_0,X_1]_\theta)
 =\Lcal_{\as(r_\theta;1)}({}^m\ell_1;[X_0,X_1]_\theta).
\]
More precisely, if
\[
 \pi_{\as(r_i;1)}(S)\le C_i\norm S
 \quad(S\in\Lcal({}^m\ell_1;X_i),\ i=0,1),
\]
then
\[
 \pi_{\as(r_\theta;1)}(T)\le C_0^{1-\theta}C_1^\theta\norm T.
\]
\end{proposition}

\begin{proof}
Put $X=[X_0,X_1]_\theta$. Fix $T\in\Lcal({}^m\ell_1;X)$ and $\varepsilon>0$.
By the isometric operator-space interpolation identity \eqref{eq:operator-interpolation}, choose
\[
 F\in\mathcal F\bigl(\Lcal({}^m\ell_1;X_0),\Lcal({}^m\ell_1;X_1)\bigr)
\]
such that $F(\theta)=T$ and boundary norm at most $\norm T+\varepsilon$.

Set $\lambda=C_0/C_1$ and replace $F$ by $z\mapsto\lambda^{z-\theta}F(z)$. This does not alter $F(\theta)$ and ensures that both endpoint estimates carry the common factor $C_0^{1-\theta}C_1^\theta$.

Let $N\in\N$ and $(x_j^{(k)})_{j=1}^N\subset\ell_1$, $1\le k\le m$. Define
\[
 G(z)=\big(F(z)(x_j^{(1)},\ldots,x_j^{(m)})\big)_{j=1}^N.
\]
This is an analytic function with values in $\ell_{r_0}^N(X_0)+\ell_{r_1}^N(X_1)$. At the left boundary, the endpoint coincidence estimate gives
\[
 \norm{G(it)}_{\ell_{r_0}^N(X_0)}
 \le C_0\norm{F(it)}\prod_{k=1}^m
 \norm{(x_j^{(k)})_{j=1}^N}_{w,1}.
\]
The analogous estimate holds on the right boundary with $C_1,r_1,X_1$. The balancing just performed and Lemma~\ref{lem:finite-vector} imply
\[
 \norm{G(\theta)}_{\ell_{r_\theta}^N(X)}
 \le C_0^{1-\theta}C_1^\theta(\norm T+\varepsilon)\prod_{k=1}^m\norm{(x_j^{(k)})_{j=1}^N}_{w,1}.
\]
Since $G(\theta)$ is the diagonal sequence determined by $T$, let $\varepsilon\downarrow0$ and then take the supremum over the finite families.
\end{proof}

\section{Multilinear Kwapie\'n's Theorem for Abstract Measure Spaces \texorpdfstring{$L_p(\mu)$}{Lp(mu)}}\label{sec:Lp-mu}

We now extend both the coincidence and the exact optimality to general measure spaces $(\Omega, \Sigma, \mu)$.

\subsection{The coincidence via complex interpolation}

The next theorem is an immediate application of Proposition~\ref{prop:diagonal}. It is recorded separately in order to state the precise scalar-valued consequence needed below.

\begin{theorem}[Multilinear Kwapie\'n's coincidence for $L_p(\mu)$]\label{thm:Lp-coincidence}
Let $(\Omega,\Sigma,\mu)$ be a measure space, $m\in\N$, and $1\le p\le\infty$. Then
\[
 \Lcal({}^m\ell_1;L_p(\mu))
 =\Lcal_{\as(r_{m,p};1)}({}^m\ell_1;L_p(\mu)),
\]
where
\[
 r_{m,p}=
 \begin{cases}
 \dfrac{2p}{mp+2p-2},&1\le p\le2,\\[6pt]
 \dfrac{2p}{mp+2},&2\le p\le\infty.
 \end{cases}
\]
\end{theorem}

\begin{proof}
The endpoint cases follow directly from the known coincidence results.
Botelho's theorem gives
\[
\Lcal({}^m\ell_1;L_1(\mu))
=
\Lcal_{\as(2/m;1)}({}^m\ell_1;L_1(\mu)).
\]
Bernardino's theorem gives
\[
\Lcal({}^m\ell_1;L_2(\mu))
=
\Lcal_{\as(2/(m+1);1)}({}^m\ell_1;L_2(\mu)).
\]
Finally, since Botelho's theorem applies to an arbitrary Banach-space
range,
\[
\Lcal({}^m\ell_1;L_\infty(\mu))
=
\Lcal_{\as(2/m;1)}({}^m\ell_1;L_\infty(\mu)).
\]

Let \(1<p<2\). Then
\[
L_p(\mu)=[L_1(\mu),L_2(\mu)]_\theta,
\qquad
\theta=2-\frac2p\in(0,1).
\]
Applying Proposition~\ref{prop:diagonal} with endpoint exponents
\(2/m\) and \(2/(m+1)\), respectively, gives
\[
\frac1r
=
(1-\theta)\frac m2+\theta\frac{m+1}{2}
=
\frac m2+1-\frac1p.
\]
Hence
\[
r=\frac{2p}{mp+2p-2}.
\]

Let \(2<p<\infty\). In this case,
\[
L_p(\mu)=[L_2(\mu),L_\infty(\mu)]_\theta,
\qquad
\theta=1-\frac2p\in(0,1).
\]
The endpoint exponents are \(2/(m+1)\) and \(2/m\), respectively.
Therefore,
\[
\frac1r
=
(1-\theta)\frac{m+1}{2}+\theta\frac m2
=
\frac m2+\frac1p,
\]
and consequently
\[
r=\frac{2p}{mp+2}.
\]
This proves the asserted formula for all \(1\le p\le\infty\).
\end{proof}
\subsection{Optimality via disjoint-support Walsh embeddings}

While the upper coincidence follows from Proposition~\ref{prop:diagonal}, the remaining issue is the matching lower bound. In a continuous or non-discrete measure space $L_p(\mu)$, canonical basis coordinates are unavailable. However, we can bridge this gap by exploiting pairwise disjoint supports.

\begin{definition}
A measure space $(\Omega, \Sigma, \mu)$ is said to have the \emph{infinite disjoint-support property} if, for every $n \in \N$, there exist pairwise disjoint sets $A_1, \ldots, A_n \in \Sigma$ such that $0 < \mu(A_j) < \infty$ for all $j=1,\ldots,n$.
\end{definition}

\begin{remark}
Assuming $(\Omega,\Sigma,\mu)$ is semi-finite (so that every set of infinite measure contains a subset of positive finite measure), the infinite disjoint-support property is satisfied in either of the following standard situations:
\begin{enumerate}
\item $\Omega$ contains infinitely many pairwise disjoint atoms of positive finite measure (e.g., discrete measure spaces with infinite support);
\item $\Omega$ contains at least one non-atomic set of positive finite measure (e.g., Lebesgue measure on $[0,1]$ or $\mathbb R^d$, or non-trivial Radon measures).
\end{enumerate}
For a semi-finite measure space, the property fails precisely when the measure algebra has only finitely many atoms of positive measure; in that case $L_p(\mu)$ is finite-dimensional.
\end{remark}

\begin{theorem}[Sharpness on general measure spaces]\label{thm:Lp-optimality}
Let $(\Omega, \Sigma, \mu)$ be a measure space possessing the infinite disjoint-support property. Let $m \in \N$ and $1 \le p \le \infty$. Then
\[
\inf\left\{r > 0: \Lcal({}^m\ell_1; L_p(\mu)) = \Lcal_{\as(r;1)}({}^m\ell_1; L_p(\mu))\right\} = r_{m,p}.
\]
\end{theorem}

\begin{proof}
Theorem~\ref{thm:Lp-coincidence} establishes the coincidence at $r = r_{m,p}$. It remains to prove that no strictly smaller exponent is admissible.

Fix $n \in \N$. By the infinite disjoint-support property, choose disjoint sets $A_1, \ldots, A_n \in \Sigma$ with $0 < \mu(A_j) < \infty$. For $1 \le p < \infty$, define the linear operator
\[
J_n: \ell_p^n \longrightarrow L_p(\mu), \qquad J_n(a_1,\ldots,a_n) := \sum_{j=1}^n a_j \frac{\mathbf{1}_{A_j}}{\mu(A_j)^{1/p}}.
\]
Since the sets $A_j$ are pairwise disjoint,
\[
\|J_n(a)\|_{L_p(\mu)}^p = \int_\Omega \left|\sum_{j=1}^n a_j \frac{\mathbf{1}_{A_j}(t)}{\mu(A_j)^{1/p}}\right|^p d\mu(t) = \sum_{j=1}^n |a_j|^p \frac{\mu(A_j)}{\mu(A_j)} = \sum_{j=1}^n |a_j|^p = \|a\|_{\ell_p^n}^p.
\]
Thus $J_n$ is a linear isometry. For $p=\infty$, setting $J_n(a) := \sum_{j=1}^n a_j \mathbf{1}_{A_j}$ gives an identical linear isometry from $\ell_\infty^n$ into $L_\infty(\mu)$.

Therefore, \(\ell_p\) is crudely finitely representable in
\(L_p(\mu)\) with isometric constant \(D=1\). Now, let $M = 2^d$, and consider the operator
\[
\widetilde A_{M,p} := J_M \circ A_{M,p} \in \Lcal({}^m\ell_1; L_p(\mu)),
\]
where $A_{M,p}$ is the Walsh test operator from Proposition~\ref{prop:Walsh-test}. Because $J_M$ is an isometry, $\|\widetilde A_{M,p}\| \le 1$ and
\[
\|\widetilde A_{M,p}(h_\alpha,\ldots,h_\alpha)\|_{L_p(\mu)} = \|A_{M,p}(h_\alpha,\ldots,h_\alpha)\|_p.
\]
Assuming $\Lcal({}^m\ell_1; L_p(\mu)) = \Lcal_{\as(r;1)}({}^m\ell_1; L_p(\mu))$, Lemma~\ref{lem:uniform-coincidence} and exactly the same asymptotic analysis as in the proof of Theorem~\ref{thm:uniform-optimality} force $r \ge r_{m,p}$.
\end{proof}

\subsection{Multilinear Kwapie\'n's theorem in other codomains}

A useful feature of the interpolation-Walsh approach is that it does not rely on the specific coordinate structure of sequence spaces. By combining the diagonal interpolation principle (Proposition~\ref{prop:diagonal}) with appropriate embeddings into the range, multilinear Kwapie\'n-type theorems with optimal exponents can be derived for several standard codomains.

\subsubsection{Hilbert-valued Bochner spaces}
Let $(\Omega, \Sigma, \mu)$ be a measure space with the infinite disjoint-support property, and let $H$ be an arbitrary Hilbert space.

\begin{corollary}[Sharp multilinear Kwapie\'n's theorem for $L_p(\mu; H)$]\label{cor:Bochner-Hilbert}
Let $m \in \N$ and $1 \le p \le \infty$. Then
\[
\Lcal({}^m\ell_1; L_p(\mu; H)) = \Lcal_{\as(r_{m,p};1)}({}^m\ell_1; L_p(\mu; H)),
\]
and the exponent $r_{m,p}$ is optimal:
\[
\inf\left\{r > 0: \Lcal({}^m\ell_1; L_p(\mu; H)) = \Lcal_{\as(r;1)}({}^m\ell_1; L_p(\mu; H))\right\} = r_{m,p}.
\]
\end{corollary}

\begin{proof}
When $p=2$, the Bochner space $L_2(\mu; H)$ is itself a Hilbert space under the inner product $\langle f, g \rangle := \int_\Omega \langle f(t), g(t)\rangle_H \, d\mu(t)$. Thus, Bernardino's theorem \cite{Bernardino} gives the coincidence at $r = 2/(m+1)$. At the endpoints $p=1$ and $p=\infty$, $L_1(\mu; H)$ and $L_\infty(\mu; H)$ are Banach spaces, so Botelho's theorem \cite{Botelho} provides the exponent $2/m$. 

By vector-valued complex interpolation (see Bergh and L\"ofstr\"om \cite[Theorem~5.1.2]{BerghLofstrom}), one has isometrically
\[
[L_1(\mu; H), L_2(\mu; H)]_\theta = L_p(\mu; H) \quad (1 < p < 2), 
\qquad 
[L_2(\mu; H), L_\infty(\mu; H)]_\theta = L_p(\mu; H) \quad (2 < p < \infty).
\]
Applying Proposition~\ref{prop:diagonal} establishes the upper coincidence with exponent $r_{m,p}$. 

For the lower bound, fix a unit vector $h_0 \in S_H$. The mapping $a \mapsto \sum_{j=1}^n a_j \mu(A_j)^{-1/p} \mathbf 1_{A_j} \otimes h_0$ is an isometric embedding of $\ell_p^n$ into $L_p(\mu; H)$. The Walsh lower estimate from Theorem~\ref{thm:Lp-optimality} therefore applies directly and shows that $r \ge r_{m,p}$.
\end{proof}

\begin{remark}
If $X$ is not isomorphic to a Hilbert space, Kwapie\'n's landmark theorem \cite{Kwapien1972} ensures that $L_2(\mu; X)$ is never a Hilbert space. In that case, Bernardino's theorem cannot be applied at the $p=2$ node, and the coincidence exponent is governed instead by the cotype of $X$.
\end{remark}

\subsubsection{Non-commutative $L_p$-spaces and Schatten ideals}
Let $\mathcal M$ be a semi-finite von Neumann algebra equipped with a faithful normal semi-finite trace $\tau$, and let $L_p(\mathcal M, \tau)$ ($1 \le p \le \infty$) be the associated non-commutative $L_p$-spaces. When $\mathcal M = \mathcal B(H)$ and $\tau = \mathrm{Tr}$, this scale recovers the classical Schatten classes $S_p(H)$ for $1 \le p < \infty$, while $L_\infty(\mathcal B(H), \mathrm{Tr}) = \mathcal B(H)$ (the compact ideal $S_\infty(H) = \mathcal K(H)$ is treated analogously via Botelho's theorem).

\begin{corollary}[Sharpness on non-commutative $L_p$]\label{cor:noncommutative}
Assume that $\mathcal M$ contains an infinite sequence of mutually orthogonal projections $(e_j)_{j=1}^\infty$ with $0 < \tau(e_j) < \infty$. Then, for every $m \in \N$ and $1 \le p \le \infty$,
\[
\inf\left\{r > 0: \Lcal({}^m\ell_1; L_p(\mathcal M, \tau)) = \Lcal_{\as(r;1)}({}^m\ell_1; L_p(\mathcal M, \tau))\right\} = r_{m,p}.
\]
\end{corollary}

\begin{proof}
Since $L_2(\mathcal M, \tau)$ is a Hilbert space under $\langle x, y \rangle = \tau(y^* x)$, and the complex interpolation scales $[L_1(\mathcal M), L_2(\mathcal M)]_\theta = L_p(\mathcal M)$ and $[L_2(\mathcal M), \mathcal M]_\theta = L_p(\mathcal M)$ hold isometrically (see Kosaki \cite{Kosaki1984} and Pisier \cite{PisierTheta}), Proposition~\ref{prop:diagonal} yields the upper coincidence. 

For the lower bound, the linear map $a \mapsto \sum_{j=1}^n a_j \tau(e_j)^{-1/p} e_j$ is an isometric embedding of $\ell_p^n$ into $L_p(\mathcal M, \tau)$ thanks to the orthogonality of the projections. The Walsh test operator analysis from Theorem~\ref{thm:Lp-optimality} applies with the same estimates and forces $r \ge r_{m,p}$.
\end{proof}

\subsubsection{Further function-space applications}
The preceding argument extends to other interpolation scales only after the relevant endpoint interpolation identities and the required uniform finite-dimensional embeddings have been verified for the specific range space. We therefore do not include Sobolev or analytic Hardy spaces in the general theorem without a separate proof of those two ingredients.

\section{Local Representability, Cotype, and Scalar Obstructions}
\label{sec:local-cotype}

In \cite[Theorem~4.7]{BPR}, Bayart, Pellegrino, and Rueda obtained
cotype-based necessary conditions for coincidence results on mixed
products. We record two refinements of the underlying finite-dimensional
argument. First, we formulate a local range-space obstruction: whenever
\(\ell_s<_f W\), the local \(\ell_s\)-structure of \(W\) yields a
necessary condition governed by \(s\). In particular, if \(W\) is infinite-dimensional and
\[
q_W:=\inf\{q\ge2: W \text{ has cotype }q\},
\]
then the Maurey--Pisier theorem recovers the cotype-based obstruction
by taking \(s=q_W\).
Second, we isolate a scalar obstruction obtained by restricting the
range to a fixed one-dimensional subspace of \(W\). This scalar part
does not require any finite-dimensionality assumption on the range and
is valid for every nonzero Banach space \(W\).

For an infinite-dimensional Banach space \(Z\), define its infimal
cotype by
\[
q_Z:=\inf\{q\ge2: Z \text{ has cotype }q\}\in[2,\infty].
\]
We use the convention \(1/\infty=0\).

We shall use the following consequence of the Maurey--Pisier theorem. If $Z$ is infinite-dimensional and has finite infimal cotype $q_Z<\infty$, then, for every $n\in\mathbb N$ and every $\varepsilon>0$, there are a subspace $Z_n\subset Z$ and an isomorphism
\[
S_n:\ell_{q_Z}^n\longrightarrow Z_n
\]
such that
\[
\|S_n\|\,\|S_n^{-1}\|\le1+\varepsilon.
\]
Thus $\ell_{q_Z}$ is finitely representable in $Z$; after fixing $\varepsilon=1$ and renormalizing, one may arrange $\|S_n\|\le1$ and $\|S_n^{-1}\|\le2$.

\begin{theorem}[Local representability and scalar obstructions]\label{thm:local-cotype}
Let $m \in \N$, $0 \le l \le m$, and $1 \le p_1, \ldots, p_l \le \infty$. Let $Z_{l+1}, \ldots, Z_m$ be infinite-dimensional Banach spaces of finite cotype. Define
\[
a := \sum_{j=1}^l \frac{1}{p_j}, \qquad b := \sum_{j=l+1}^m \frac{1}{q_{Z_j}}.
\]
When \(l=0\), we use the convention \(a=0\), and the coefficient
space occurring below is understood to be \(\ell_\infty^n\).
Let $W \ne \{0\}$ be an arbitrary Banach space. If
\begin{equation}\label{eq:mixed-coinc}
\Lcal(\ell_{p_1},\ldots,\ell_{p_l},Z_{l+1},\ldots,Z_m; W) = \Lcal_{\as(r;1)}(\ell_{p_1},\ldots,\ell_{p_l},Z_{l+1},\ldots,Z_m; W),
\end{equation}
then the following hold:
\begin{enumerate}
\item[\textup{(i)}] \textbf{Universal scalar bound:}
\begin{equation}\label{eq:bound-scalar}
\frac1r \le \max\{1-a, 0\} + a + b.
\end{equation}
\item[\textup{(ii)}] \textbf{Local range bound:} If $\ell_s <_f W$ for some $1 \le s \le \infty$, then
\begin{equation}\label{eq:bound-local-s}
\frac1r \le \max\left\{\frac1s - a, 0\right\} + a + b.
\end{equation}
\end{enumerate}
\end{theorem}

\begin{proof}
We first prove (ii). Let $D \ge 1$ be such that for every $n \in \N$ there is an isomorphism $J_n: \ell_s^n \to W_n \subset W$ with $\|J_n\| \le 1$ and $\|J_n^{-1}\| \le D$. Let $w_i := J_n(e_i) \in W$ ($i=1,\ldots,n$).

By Maurey--Pisier, for each $j = l+1,\ldots,m$, choose $S_{j,n}: \ell_{q_{Z_j}}^n \to Z_{j,n} \subset Z_j$ with $\|S_{j,n}\| \le 1$ and $\|S_{j,n}^{-1}\| \le 2$. Set $y_{j,i} := S_{j,n}(e_i)$, and let $y_{j,i}^* \in Z_j^*$ be Hahn--Banach extensions of the coordinate functionals, so that $y_{j,i}^*(y_{j,k}) = \delta_{ik}$.

Define the $m$-linear operator $T_n$ by
\[
T_n(x^{(1)},\ldots,x^{(l)},z_{l+1},\ldots,z_m) := \sum_{i=1}^n x_i^{(1)} \cdots x_i^{(l)} \prod_{j=l+1}^m y_{j,i}^*(z_j) \, w_i.
\]
Since the Hahn--Banach extensions satisfy
\[
\|y_{j,i}^*\|\le \|S_{j,n}^{-1}\|\le2,
\]
H\"older's inequality, together with the finite-dimensional inclusion
\(\ell_{1/a}^n\hookrightarrow\ell_s^n\) when \(a>0\), yields
\[
\|T_n\|\le 2^{m-l}n^{\delta_s},
\qquad
\delta_s:=\max\left\{\frac1s-a,0\right\}.
\]
When \(a=0\), the same estimate follows by interpreting the
coefficient space as \(\ell_\infty^n\).
By Lemma~\ref{lem:uniform-coincidence}, there is $C > 0$ such that $\pi_{\as(r;1)}(T_n) \le C \|T_n\|$. Evaluating $T_n$ on $(e_i, \ldots, e_i, y_{l+1,i}, \ldots, y_{m,i})_{i=1}^n$ yields $T_n(\cdots) = w_i$. Hence
\[
D^{-1} n^{1/r} \le \left(\sum_{i=1}^n \|w_i\|_W^r\right)^{1/r} \le C \|T_n\| \prod_{j=1}^l \|(e_i)_{i=1}^n\|_{w,1;\ell_{p_j}} \prod_{j=l+1}^m \|(y_{j,i})_{i=1}^n\|_{w,1;Z_j}.
\]
Since
\[
\|(e_i)_{i=1}^n\|_{w,1;\ell_{p_j}}=n^{1/p_j},
\]
and, for every \(j=l+1,\ldots,m\),
\[
\|(y_{j,i})_{i=1}^n\|_{w,1;Z_j}
\le n^{1/q_{Z_j}},
\]
we obtain
\[
D^{-1}n^{1/r}
\le C\,2^{m-l}n^{\delta_s+a+b}.
\]
Indeed, if \(\varphi\in B_{Z_j^*}\), then
\[
\sum_{i=1}^n |\varphi(y_{j,i})|
=
\sum_{i=1}^n |S_{j,n}^*(\varphi)(e_i)|
\le n^{1/q_{Z_j}}
\|S_{j,n}^*(\varphi)\|_{\ell_{q_{Z_j}'}^n}
\le n^{1/q_{Z_j}}.
\]
Letting $n \to \infty$ forces $1/r \le \delta_s + a + b$, which proves \eqref{eq:bound-local-s}.

For (i), pick any fixed unit vector $w_0 \in S_W$ and define $T_n$ as above with $w_i = w_0$ for all $i$. Since \(\|y_{j,i}^*\|\le2\), the scalar range computation yields
\[
\|T_n\|
\le
2^{m-l}n^{\max\{1-a,0\}}.
\] Evaluating on the same diagonal gives $T_n(\cdots) = w_0$, so the sum of norms is $n^{1/r}$. The exact same weak-norm estimates yield $n^{1/r} \le C 2^{m-l} n^{\max\{1-a,0\} + a + b}$, proving \eqref{eq:bound-scalar}.
\end{proof}

\begin{remark}[Scope and comparison of lower bounds]
It is instructive to compare the bounds provided by Theorem~\ref{thm:local-cotype} and the Walsh method of Section~\ref{sec:Walsh}:
\begin{enumerate}
\item In the special case of $(\ell_1)^m$ (where $l=m$, $p_j=1$, and $a=m$), Theorem~\ref{thm:local-cotype}(ii) with $\ell_p <_f W$ gives
\[
\frac1r \le \max\left\{\frac1p - m, 0\right\} + m = m \iff r \ge \frac1m.
\]
On the other hand, our Walsh construction gives
\[
\frac1r \le \frac m2 + 1 - \frac1p \iff r \ge \frac{2p}{mp+2p-2}.
\]
Since
\[
\frac m2+1-\frac1p<m
\qquad
(m\ge2,\ 1\le p\le2),
\]
the Walsh estimate yields a strictly smaller upper bound for \(1/r\),
or equivalently a strictly larger lower bound for \(r\). Hence it is
strictly stronger than the local representability obstruction in this
special setting.
\item However, the true value of Theorem~\ref{thm:local-cotype} lies in its generality: it provides lower bounds for arbitrary mixed products $\ell_{p_1} \times \cdots \times \ell_{p_l} \times Z_{l+1} \times \cdots \times Z_m$, where the absence of a group structure precludes the definition of Walsh characters.
\end{enumerate}
\end{remark}

\section{Sharp Kwapie\'n's Theorem for Blocks}\label{sec:blocks}

The block-summing framework, developed by N\'u\~nez-Alarc\'on, Santos, and the third author \cite{NSS} (see also \cite{BlockSumming}), provides a continuous scale connecting diagonal summability (where exponents depend on $m$) and multiple summability (where exponents freeze at their linear values). 

Let $\mathcal I=\{I_1,\ldots,I_k\}$ be a partition of $\{1,\ldots,m\}$ into non-empty pairwise disjoint subsets, and let
\[
\Omega_{\mathcal I} := \big\{(j_1,\ldots,j_m) \in \N^m : j_r = j_s \text{ whenever } r, s \in I_\nu \text{ for some } 1 \le \nu \le k\big\}.
\]
An operator $T \in \Lcal({}^m\ell_1; F)$ is called \emph{block $(r;1)$-summing with respect to $\mathcal I$} if there exists $C \ge 0$ such that, for every $N \in \N$ and all families $(x_j^{(q)})_{j=1}^N \subset \ell_1$ ($1 \le q \le m$),
\[
\left(\sum_{(j_1,\ldots,j_m) \in \Omega_{\mathcal I} \cap \{1,\ldots,N\}^m} \norm{T(x_{j_1}^{(1)},\ldots,x_{j_m}^{(m)})}_F^r\right)^{1/r} \le C \prod_{q=1}^m \norm{(x_j^{(q)})_{j=1}^N}_{w,1}.
\]
The class of all such operators is denoted by $\Lcal_{\mathcal I, \as(r;1)}({}^m\ell_1; F)$. 

When $\mathcal I = \{\{1,\ldots,m\}\}$, this recovers diagonal absolute summability $\Lcal_{\as(r;1)}({}^m\ell_1; F)$, whereas $\mathcal I = \{\{1\},\ldots,\{m\}\}$ recovers multiple summability $\Lcal_{\mult(r;1)}({}^m\ell_1; F)$.

In \cite[Theorem~4.2]{NSS}, the block coincidence was established for range space $\ell_p$ with the exponent
\begin{equation}\label{eq:block-exponent}
r(\mathcal I, p) = \begin{cases}
\dfrac{2p}{np+2p-2}, & 1 \le p \le 2,\\[8pt]
\dfrac{2p}{np+2}, & 2 \le p \le \infty,
\end{cases}
\qquad \text{where } n := \min_{1 \le j \le k} |I_j|,
\end{equation}
and its optimality was proved for $2 \le p \le \infty$. However, the optimality for $1 \le p < 2$ was left open. Combining our Walsh construction with the block-reduction equivalence from \cite{NSS}, this problem is now settled in full generality.
\begin{theorem}[Sharp Kwapie\'n's theorem for arbitrary blocks]\label{thm:sharp-blocks}
Let $m \in \N$, let $\mathcal I=\{I_1,\ldots,I_k\}$ be a partition of $\{1,\ldots,m\}$, and set $n := \min_{1 \le j \le k} |I_j|$. Let $1 \le p \le \infty$, and let $F$ be one of the range spaces for which the diagonal coincidence and the Walsh lower embedding have been established above: $\ell_p$, $L_p(\mu)$ (with the infinite disjoint-support property), $L_p(\mathcal M,\tau)$ under the hypothesis of Corollary~\ref{cor:noncommutative}, or $L_p(\mu;H)$ under the hypothesis of Corollary~\ref{cor:Bochner-Hilbert}. Then
\[
\Lcal({}^m\ell_1; F) = \Lcal_{\mathcal I, \as(r(\mathcal I, p); 1)}({}^m\ell_1; F),
\]
and the exponent $r(\mathcal I, p)$ defined in \eqref{eq:block-exponent} is optimal:
\[
\inf\left\{r > 0 : \Lcal({}^m\ell_1; F) = \Lcal_{\mathcal I, \as(r; 1)}({}^m\ell_1; F)\right\} = r(\mathcal I, p).
\]
\end{theorem}

\begin{proof}
By \cite[Lemma~4.1]{NSS}, for any Banach space $F$ and any exponent $r > 0$, one has the exact equivalence:
\[
\Lcal({}^m\ell_1; F) = \Lcal_{\mathcal I, \as(r; 1)}({}^m\ell_1; F) \iff \Lcal({}^n\ell_1; F) = \Lcal_{\as(r; 1)}({}^n\ell_1; F),
\]
where $n = \min_{1 \le j \le k} |I_j|$. 

The coincidence at $r = r(\mathcal I, p) = r_{n,p}$ follows from Theorem~\ref{thm:Lp-coincidence} (together with Corollaries~\ref{cor:Bochner-Hilbert} and~\ref{cor:noncommutative}, as appropriate). Conversely, if the coincidence holds for some $r > 0$, this equivalence forces $\Lcal({}^n\ell_1; F) = \Lcal_{\as(r; 1)}({}^n\ell_1; F)$. Our Walsh optimality theorem (Theorem~\ref{thm:uniform-optimality} and Theorem~\ref{thm:Lp-optimality}) applied to $n$-linear operators directly yields $r \ge r_{n,p} = r(\mathcal I, p)$.
\end{proof}

\section*{Acknowledgments}

This work forms part of the first author's Ph.D. thesis, written under
the supervision of Professor Daniel N\'u\~nez-Alarc\'on. The authors
thank Professor N\'u\~nez-Alarc\'on for valuable comments and
discussions.

\end{document}